\documentclass[12pt,amssymb,amscd]{amsart}
\usepackage{color}
\usepackage{amssymb}
\usepackage{amsmath}
\usepackage{amsthm} 
\usepackage{mathrsfs}                                                           
\usepackage[all]{xy}
\usepackage{wrapfig} 
\usepackage{verbatim}
\usepackage{appendix}
\usepackage{color}

\def\2{{1\over 2}}

\newcommand{\rf}[1]{(\ref{#1})}
\def\b{\bar}

\renewcommand{\t}{\tilde}
\newcommand{\p}{\partial}

\def\cE{\mathcal{E}}
\def\cF{\mathcal{F}}

\def\R{\mathbb{R}}
\def\a{\alpha}
\def\D{\Delta}

\def\inv{^{-1}}
\def\<{\langle}
\def\>{\rangle}
\def\+{\dagger}
\def\tab{\;\;\;\;\;\;}
\def\ox{\otimes}

\renewcommand{\i}{\mathrm{i}}

\newtheorem{Thm}{Theorem}[section]
\newtheorem{Lem}[Thm]{Lemma}
\newtheorem{Prop}[Thm]{Proposition}
\newtheorem{Cor}[Thm]{Corollary}
\newtheorem{Rem}[Thm]{Remark}
\newtheorem{Def}[Thm]{Definition}

\def\R{\mathbb{R}}
\def\Q{\mathbb{Q}}

\def\C{\mathbb{C}}
\def\Z{\mathbb{Z}}

\def\to{\longrightarrow}

\def\cE{\mathcal{E}}
\def\cF{\mathcal{F}}

\def\cK{\mathcal{K}}

\def\cN{\mathcal{N}}
\def\cP{\mathcal{P}}

\def\cU{\mathcal{U}}

\def\cW{\mathcal{W}}
\def\a{\alpha}
\def\b{\beta}

\def\c{\gamma}
\def\D{\Delta}

\def\s{\sigma}
\def\t{\tilde}
\def\W{\Omega}

\def\ze{\zeta}

\def\sl{\mathfrak{sl}}
\def\g{\mathfrak{g}}
\def\gl{\mathfrak{gl}}

\def\osp{\mathfrak{osp}}

\def\ox{\otimes}
\def\o+{\oplus}
\def\bo+{\bigoplus}

\def\p[#1,#2]{\phi_{#1,#2}}
\def\what[#1]{\widehat{#1}}

\def\bi{\textbf{i}}

\def\oo{\infty}
\def\=>{\Longrightarrow}

\def\<{\langle}
\def\>{\rangle}
\def\corr{\longleftrightarrow}
\def\^{\wedge}
\def\+{\dagger}

\def\iff{\Longleftrightarrow}

\def\inv{^{-1}}

\def\over[#1]{\overline{#1}}
\def\vec[#1]{\overrightarrow{#1}}
\def\xto[#1]{\xrightarrow{#1}}
\def\dd[#1,#2]{\frac{d#1}{d#2}}
\def\del[#1,#2]{\frac{\partial #1}{\partial #2}}
\def\tab{\;\;\;\;\;\;}

\newcommand{\til}[1]{\widetilde{#1}}
\newcommand{\veca}[2][cccccccccccccccccccccccccccccccccccccccccc]{\left(\begin{array}{#1}#2 \\ \end{array} \right)}

\newcommand{\Eq}[1]{\begin{align}#1\end{align}}
\newcommand{\Eqn}[1]{\begin{align*}#1\end{align*}}

\begin{document}
\title{Supersymmetry and the Modular Double}
\author{Ivan Chi-Ho Ip}
\author{Anton M. Zeitlin}
\address{\newline
Ivan Chi-Ho Ip,\newline
Kavli Institute for the Physics and Mathematics of the Universe,\newline
University of Tokyo,\newline
Kashiwa, Chiba,\newline
277-8583 Japan\newline
ivan.ip@ipmu.jp\newline
http://member.ipmu.jp/ivan.ip}
\address{ 
\newline 
Anton M. Zeitlin,\newline
Department of Mathematics,\newline
Columbia University,\newline
Room 509, MC 4406,\newline
2990 Broadway,\newline
New York, NY 10027,\newline
zeitlin@math.columbia.edu,\newline
http://math.columbia.edu/$\sim$zeitlin \newline
http://www.ipme.ru/zam.html  }

\maketitle
\begin{abstract}
A counterpart of the modular double for quantum superalgebra
$\cU_q(\osp(1|2))$ is constructed by means of 
supersymmetric quantum mechanics. We also construct the $R$-matrix operator acting in the corresponding representations, which is expressed via quantum dilogarithm.
\end{abstract}
\section{Introduction}
The so-called modular double was introduced in 1999 for the quantum algebra  $\cU_q(\sl(2, \mathbb{R}))$ in \cite{faddeev}. The modular double is a certain representation in $L^2(\mathbb{R})$ of two quantum algebras 
$\cU_q(\sl(2, \mathbb{R}))$ and $\cU_{\t q}(\sl(2, \mathbb{R}))$, so that 
$|q|=|\t q|=1$, $q=e^{\pi ib^2}$, $\t q=e^{\pi ib^{-2}}$. This representation is constructed via a certain 
"free field realization" of  $\cU_q(\sl(2, \mathbb{R}))$ by means of the Weyl algebra. The resulting generators of both 
$\cU_q(\sl(2, \mathbb{R}))$ and $\cU_{\t q}(\sl(2,\mathbb{R}))$ turn out to be unbounded positively-defined essentially self-adjoint operators in 
$L^2(\mathbb{R})$. 

One of the reasons why the modular double is important, is as follows. 
There is a well known result of rational conformal field 
theory, that the braided tensor categories of unitary representations of compact quantum groups have their equivalent counterparts in the category of representations of the corresponding WZW models, and by the Drinfeld-Sokolov reduction, in the category of representations of $W$-algebras. 
It appears that the representation of the modular double of $\cU_q(\sl(2,\mathbb R))$ plays the same role for the Liouville theory. For example, 3j-symbols for the tensor product of modular double representations \cite{PT2} appear in the fusion product for the Liouville vertex operators \cite{teschner}. 
It is also expected that there are generalizations of the modular double representation to the higher rank, which should be related to the braided tensor categories from relevant Toda field theories. There were several attempts to construct such representations \cite{fip}, \cite{gerasimov}, \cite{Ip2}, \cite{Ip3}, \cite{kashaev}. 

Three years ago, Liouville and related Toda theories have drawn a lot of attention in the context of the so-called AGT correspondence \cite{agt}, which is a correspondence between two dimensional models and 4-dimensional gauge theory. 
Recently it was shown that the supersymmetric Toda theories, especially $\cN=1$ SUSY Liouville theory are also very important in this context \cite{belavin}, \cite{bonelli}. 
It is natural to expect that the corresponding quantum superalgebras also possess the modular double representation, which 
will characterize the fusion products as well as in the bosonic case. 
The simplest supersymmetric Toda-like theory is $\cN=1$ SUSY Liouville theory, which is related to the superalgebra $\osp(1|2)$. This superalgebra  plays an important role in the classification of semisimple superalgebras: it is an analogue of $\sl(2)$ subalgebra for odd (black) roots in the higher rank case. 

In this note we construct an analogue of the modular double representation for the real form of $\cU_q(\osp(1|2))$, whose tensor category, as we hope, will be related to fusion products of vertex operators in $\cN=1$ SUSY Liouville theory.

The structure of the paper is as follows. In Section 2 we remind the structure of the modular double of $\cU_q(\sl(2,\mathbb R))$ including properties of quantum dilogarithm, which is relevant for the construction of the $R$-matrix. Section 3 is devoted to the construction of the analogue of modular double for $\cU_q(\osp(1|2))$, while the $R$-matrix is given in Section 4. In Section 5 we outline possible directions of future study.

\section{Reminder of the modular double.} 
\noindent{\bf 2.1.  Construction of Representations.} Let $q=e^{\pi ib^2}$ where ${0<b<1}$ and $b^2\in\R\setminus\Q$. Let us consider the operators $p$, $x$ satisfying the Heisenberg commutation relation: $[p,x]=\frac{1}{2\pi i}$. Then the following operators:
\begin{eqnarray}
U=e^{2\pi bx},\quad  V=e^{2\pi b p}, \quad 
\end{eqnarray}
satisfy the quantum plane commutation relation: 
\Eq{UV=q^2VU.}
Then, one can construct a realization of the generators of $\cU_q(\sl(2, \mathbb{R}))$ by means of just $U,V$ and some real parameter $Z$:
\begin{eqnarray}\label{efk}
E=i\frac{V+U^{-1}Z}{q-q^{-1}},\quad F=i\frac{U+V^{-1}Z^{-1}}{q-q^{-1}}, \quad K=q^{-1}UV, 
\end{eqnarray}
in other words
\begin{eqnarray}
[E,F]=\frac{K-K^{-1}}{q-q^{-1}}, \quad KE=q^2EK, \quad KF=q^{-2}FK.
\end{eqnarray}
Therefore, considering the standard representation of the Heisenberg algebra ($p=\frac{1}{2\pi i}\frac{d}{dx}$) on 
$L^2(\mathbb{R})$, one obtains a representation $\cP_z$ of $\cU_q(\sl(2, \mathbb{R}))$ on $L^2(\mathbb{R})$ by \emph{positive operators} \cite{schmudgen}. More precisely, as one can see from the form of the generators, they are manifestly Hermitian, and it can be shown that $E, F, K$ are unbounded positive essentially self-adjoint operators on $L^2(\mathbb{R})$, with a dense core given by the subspace
\Eq{\label{WW}\cW=span\{e^{-\a x^2+\b}P(x): \a\in\R_{>0}, \b\in \C\},}
where $P(x)$ is polynomial in $x$. The observation of \cite{faddeev} is that one can consider the dual pair of generating elements
\begin{eqnarray}
\t U=U^{\frac{1}{b^2}},\quad \t V=V^{\frac{1}{b^2}},
\end{eqnarray}
which satisfy the relation $\tilde U\tilde V=\tilde q^2\tilde V\tilde U$, where $\tilde q=e^{\pi i b^{-2}}$, and therefore define another representation of $\cU_{\tilde q}(\sl(2, \mathbb{R}))$ on the space $L^2(\mathbb{R})$. The corresponding generators $\tilde E$, 
$\tilde F$, $\tilde K$ commute with $E, F, K$ on the dense set $\cW$ in $L^2(\mathbb{R})$. However, their spectral projections do not commute. 
It is important to mention that the corresponding Hopf algebra structures are compatible (we see that $\t E, \t F, \t K$ depend on  $E, F, K$) . To prove it, one needs the following proposition (see \cite{BT}, \cite{Vo}).\\

\noindent{\bf Proposition 2.1.} {\it Let $A$ and $B$ be self-adjoint operators, such that $[A,B]=2\pi i$. Then the following formula holds:
\begin{eqnarray}
(u+v)^{\frac{1}{b^2}}=u^{\frac{1}{b^2}}+v^{\frac{1}{b^2}},
\end{eqnarray}
if $u=e^{bA}$ and $v=e^{bB}$.}\\

Then, if we denote by $$e:=(2\sin{\pi b^2})E, \tab f:=(2\sin{\pi b^2})F,$$ and similarly 
$$\t e:=(2\sin{\pi b^{-2}})\tilde E,\tab \t f:=(2\sin{\pi b^{-2}})\tilde F,$$ one obtains that 
\begin{eqnarray}\label{powers}
\t e=e^{\frac{1}{b^2}},\quad  \t f=f^{\frac{1}{b^2}}.
\end{eqnarray}
Therefore, if we look at the standard coproduct applied to the generators of $\cU_q(\sl(2, \mathbb{R}))$:
\begin{eqnarray}\label{sl2coproduct}
&&\Delta(e)=e\otimes K+ 1\otimes e,\\
&&\Delta(f)=f\otimes 1+ K^{-1}\otimes f,\nonumber\\
&&\Delta(K)=K\otimes K,\nonumber
\end{eqnarray}
considered as operators on $L_2(\mathbb{R}\times \mathbb{R})$ one obtains via Proposition 2.1 that 
\begin{eqnarray}
&&\Delta(\t e)=\t e\otimes \tilde K+  1\otimes \t e,\\
&&\Delta(\t f)=\t f\otimes 1+ \tilde K^{-1}\otimes \t f,\nonumber\\
&&\Delta(\tilde K)=\tilde K\otimes \tilde K,\nonumber
\end{eqnarray}
reproducing the coproduct structure of $\cU_{\tilde q}(\sl(2, \mathbb{R}))$. Therefore, the above results can be summarized in the following theorem.\\

\noindent {\bf Theorem 2.1.} {\it 
i)The self-adjoint operators $E,F, K$ from \rf{efk} generate a unitary representation of 
the Hopf algebra $\cU_q(\sl(2,\mathbb{R}))$ on $L^2(\mathbb{R})$. \\
ii) The self-adjoint operators $\t E, \t F, \t K$, related to $E,F, K$ via substitution $b\to \frac{1}{b}$, 
generate Hopf algebra $\cU_{\t q}(\sl(2,\mathbb{R}))$ where the coproduct is induced by the one of 
$\cU_q(\sl(2,\mathbb{R}))$ via the formulas \rf{powers}.}\\

In the following we will refer to $\cU_{\t q}(\sl(2,\mathbb{R}))$ as a $modular$ $dual$ of $\cU_{q}(\sl(2,\mathbb{R}))$. 
We note here, that on a dense subset of $L^2(\mathbb{R})$ it makes sense to talk about the action of the tensor product 
\Eq{\cU_{q\tilde{q}}(\sl(2,\R)):=\cU_{q}(\sl(2,\mathbb{R}))\otimes \cU_{\t q}(\sl(2,\mathbb{R})),} however it is not true in general because the corresponding self-adjoint operators do not commute, as we mentioned before. \\

\noindent{ \bf2.2. Quantum dilogarithm and its properties.}  In order to define the universal $R$ matrix in the case of $\cU_q(\sl(2,\R))$, one needs to introduce two special functions $G_b(x)$ and $g_b(x)$, called the quantum dilogarithm functions. In this subsection, let us recall the definition and some properties \cite{BT, Ip1, PT2} that are needed in the calculations in this paper.

\begin{Def} Let $Q=b+b\inv$. The quantum dilogarithm function $G_b(x)$ is defined on ${0< Re(z)< Q}$ by the integral formula
\Eq{\label{intform} G_b(x)=\over[\ze_b]\exp\left(-\int_{\W}\frac{e^{\pi tz}}{(e^{\pi bt}-1)(e^{\pi b\inv t}-1)}\frac{dt}{t}\right),}
where \Eq{\ze_b=e^{\frac{\pi \bi }{2}(\frac{b^2+b^{-2}}{6}+\frac{1}{2})},}
and the contour goes along $\R$ with a small semicircle going above the pole at $t=0$. This can be extended meromorphically to the whole complex plane with poles at $x=-nb-mb\inv$ and zeros at $x=Q+nb+mb\inv$, for $n,m\in\Z_{\geq0}$.
\end{Def}

\begin{Def}\label{Defgb} The function $g_b(x)$ is defined by
\Eq{g_b(x)=\frac{\over[\ze_b]}{G_b(\frac{Q}{2}+\frac{\log x}{2\pi i b})},}
where $\log$ takes the principal branch of $x$.
\end{Def}

The function $g_b(x)$ can also be written as an integral formula with $G_b$ as its kernel:

\begin{Lem}\label{FT} \cite[(3.31), (3.32)]{BT} We have the following Fourier transformation formula:
\Eq{\int_{\R+i 0} \frac{e^{-\pi i  t^2}}{G_b(Q+i t)}x^{i b\inv t}dt=g_b(x), }
\Eq{\int_{\R+i 0} \frac{e^{-\pi Qt}}{G_b(Q+i t)}x^{i b\inv t}dt=g_b^*(x) ,} where $x>0$ and the contour goes above the pole
at $t=0$.
\end{Lem}

These functions satisfy the following important properties that will be needed in our analysis.

\begin{Prop} We have the following properties for $G_b(x)$.

Self-duality:
\Eq{G_b(x)=G_{b\inv}(x);\label{selfdual}}

Functional equations: \Eq{\label{funceq}G_b(x+b^{\pm 1})=(1-e^{2\pi i b^{\pm 1}x})G_b(x);}

Reflection property:
\Eq{\label{reflection}G_b(x)G_b(Q-x)=e^{\pi i x(x-Q)};}

Complex conjugation: \Eq{\overline{G_b(x)}=\frac{1}{G_b(Q-\bar{x})},\label{Gbcomplex}}
in particular \Eq{\label{gb1}\left|G_b(\frac{Q}{2}+i x)\right|=1 \mbox{ for $x\in\R$};}

\label{asymp} Asymptotic properties:
\Eq{G_b(x)\sim\left\{\begin{array}{cc}\bar{\ze_b}&Im(x)\to+\oo\\\ze_b
e^{\pi i x(x-Q)}&Im(x)\to-\oo\end{array}.\right.}
\end{Prop}

In particular, using Definition \ref{Defgb}, these properties induce the following properties on $g_b(x)$:
\begin{Cor}\label{Propgb}We have the following properties for $g_b(x)$ for $x>0$:

Self duality:
\Eq{\label{selfdualg}g_b(x)=g_{b\inv}(x^{\frac{1}{b^2}});}

Unitarity:
\Eq{|g_b(x)|=1,\label{gbunitary}}
in particular, $g_b(X)$ is a unitary operator when $X$ is positive;

Asymptotic properties:
$g_b(z)$ is bounded for $\arg z>0$. More precisely, for $x>0$,
\Eq{|g_b(e^{\pi it}x)|\sim \left\{\begin{array}{cc}x^{-\frac{t}{2b^2}}&x\to+\oo\\
1&x\to0\end{array},\right. \forall t\in\R.}
\end{Cor}

Finally, we have the quantum exponential relations as well as the pentagon relation:
\begin{Lem}\label{qsum}If $UV=q^2VU$ where $U,V$ are positive self adjoint operators, then
\begin{eqnarray}
\label{qsum0}g_b(U)g_b(V)&=&g_b(U+V),\\
\label{qsum1}g_b(U)^*Vg_b(U) &=& q\inv UV+V,\\
\label{qsum2}g_b(V)Ug_b(V)^*&=&U+q\inv UV.
\end{eqnarray}
Note that \eqref{qsum0} and \eqref{qsum1} together imply the pentagon relation
\Eq{\label{qpenta}g_b(V)g_b(U)=g_b(U)g_b(q\inv UV)g_b(V).}\\
\end{Lem}


\noindent{\bf 2.3. Universal $R$ operator for $\cU_{q\tilde{q}}(\sl(2,\R))$}.
In the case of $\cU_{q\tilde{q}}(\sl(2,\R))$, the universal $R$ operator is realized (see e.g. \cite{BT},  \cite{faddeev}) as a unitary operator on $P_{z_1}\ox \cP_{z_2}$ such that it gives a braiding structure of the representation of tensor product. More precisely, it satisfies the following relations:
\begin{itemize}
\item The braiding relation
\Eq{\D'(X)R:=(\s \circ \D)(X)R=R\D(X), \tab \s(x\ox y)=y\ox x\label{br};}
\item The quasi-triangularity relations
\Eq{(\D \ox id)(R)=&R_{13}R_{23}\label{qt1},\\
(id\ox\D)(R)=&R_{13}R_{12}\label{qt2}.
}\end{itemize}
Here the coproduct $\D$ acts on $R$ in a natural way on the generators, and we have also used the standard leg notation. These together imply the Yang-Baxter equation
\Eq{R_{12}R_{13}R_{23}=R_{23}R_{13}R_{12}.\label{YB}}

An explicit expression of the $R$-operator is computed in \cite{BT}. It is given formally by
\Eq{R=q^{\frac{H\ox H}{2}} g_b(e\ox f),\label{RR}}
where we recall
\Eq{e:=2\sin(\pi b^2)E,\tab f:=2\sin(\pi b^2)F, \tab K:=q^H.}

The operator $R$ acts naturally on $P_{z_1}\ox P_{z_2}$ by means of the positive representations \eqref{efk}. Note that the argument $e\ox f$ inside the quantum dilogarithm $g_b$ is positive, which makes the expression a well-defined operator. In fact it is clear that $R$ acts as a unitary operator by \eqref{gbunitary} of the properties of  $g_b(x)$. Furthermore, by the transcendental relations \eqref{powers} and self-duality \eqref{selfdualg} of $g_b$, the expression \eqref{RR} is invariant under the change of $b\corr b\inv$:
\Eq{R=\til{R}:=\til{q}^{\frac{\til{H}\ox \til{H}}{2}} g_{b\inv}(\til{e}\ox \til{f}).}
Hence in fact it simultaneously serves as the $R$-operator of the modular double $\cU_{q\til{q}}(\sl(2,\R))$.

The properties as an $R$-operator are equivalent to certain functional equations for the quantum dilogarithm $g_b$. While the quasi-triangularity relations \eqref{qt1}-\eqref{qt2} are equivalent to the quantum exponential relation \eqref{qsum0}, the braiding relation 
$$\D'(X)R=R\D(X),$$ proved in \cite{BT}, 
implies the following

\begin{Lem} \label{rank1prop}We have
\Eq{
\D'(X)R&=R\D(X)\nonumber\\
\iff\D'(e)q^{\frac{1}{2}H\ox H}g_b(e\ox f)H\ox H&=q^{\frac{1}{2}H\ox H}g_b(e\ox f)\D(e)\nonumber\\
\iff(e\ox K\inv +1\ox e)g_b(e\ox f)&=g_b(e\ox f)(e\ox K+1\ox e),
}
and similarly
\Eq{
(f\ox 1 +K\ox f)g_b(e\ox f)&=g_b(e\ox f)(f\ox 1+K\inv\ox f).
}
\end{Lem}

\section{Modular double for $\cU_q(\osp(1|2))$ }

\noindent{\bf 3.1. Superalgebra $\gl(1|1)$: notations.} 
Let us consider the two-dimensional Clifford algebra $Cl_2$, generated by two elements $\xi$, $\eta$, such that 
\begin{eqnarray}
\xi^2=1,\quad  \eta^2=1, \quad \eta\xi+\xi\eta=1.
\end{eqnarray}
Every element in $Cl_2$ is a linear combination of $1, \xi, \eta$ and $\xi\eta$. 
We can introduce $\Z_2$ grading on the elements, so that  $\xi$, $\eta$ are odd, transforming $Cl_2$ into superalgebra $\cU(\gl(1|1))$.
Let us remind the matrix notation for the two-dimensional irreducible representation of two-dimensional Clifford algebra (and therefore $\cU(\gl(1|1))$) in $\mathbb{C}^{1|1}$:
\Eq{\xi=\veca{0&1\\1&0},\tab \eta=\veca{0&i\\-i&0},\tab i\eta\xi=\veca{-1&0\\0&1}.}
We will often deal with tensor products of superalgebra representations, so let us also remind how the tensor product of two elements of $\cU(\gl(1|1))$ is written using $4\times 4 $ matrix in terms of superalgebraic basis,
$$\veca{a&b\\c&d}\ox \veca{w&x\\y&z}=\veca{aw&ax&bw&bx\\ay&az&by&bz\\cw&-cx&dw&-dx\\-cy&cz&-dy&dz}.$$
A very important element for us of this form is $\xi\otimes \eta$:
\Eq{\xi\ox\eta=\veca{0&0&0&i\\0&0&-i&0\\0&-i&0&0\\i&0&0&0},}
so that $(\xi\ox\eta)^2=-1$. It can be reduced to diagonal form:
\Eq{P^*(\xi\ox\eta)P=\veca{-i&0&0&0\\0&-i&0&0\\0&0&i&0\\0&0&0&i},} where
\Eq{P=\frac{1}{\sqrt{2}}\veca{-1&0&1&0\\0&1&0&-1\\0&1&0&1\\1&0&1&0},}
or in terms of odd generators,
\Eq{P=\frac{1}{2\sqrt{2}}((\xi\ox\xi-i\eta\ox \xi+1\ox i\eta\xi-i\eta\xi\ox i\eta\xi)-(1\ox \xi+\xi\ox i\eta\xi+\i\eta\ox i\eta\xi+i\eta\xi\ox \xi)).}\\

\noindent {\bf 3.2. Construction of the generators.} Let $q_*=e^{\pi ib_*^2}$ where ${b_*^2=b^2+\frac{1}{2}}$, so that  $\frac{1}{2}<b_*^2<1$, and
\Eq{q_*^2=e^{2\pi ib_*^2}=e^{\pi i(2b^2+1)}=-q^2,\tab q_*=iq.}

Let us consider elements  $\cE,\cF,\cK$ such that
\Eq{[\cE,\cF]_+:=\cE\cF+\cF\cE=i\frac{\cK-\cK^{-1}}{q-q^{-1}},}
\Eq{\cK\cE=q^2\cE\cK, \tab \cK\cF=q^{-2}\cF\cK.}

They generate a Hopf algebra denoted as $\cU_q(\osp(1|2))$, so that the coproduct is given by:
\Eq{\D(\cE)=\cE\ox \cK+1\ox \cE, \tab \D(\cF)=\cK\inv\ox \cF+\cF\ox 1,\tab \D(\cK)=\cK\ox \cK.}

Let $E,F,K$ denote the generators of $\cU_{q_*}(\sl(2,\R))$, i.e.
\Eq{EF-FE=\frac{K-K\inv}{q_*-q_*\inv},}
\Eq{KE=q_*^2EK=-q^2EK,\tab KF=q_*^{-2}FK=-q^{-2}FK.}

Let us introduce two anti-commuting supersymmetric generators and the involution element from $\cU(\gl(1|1))\otimes \cU_{q_*}(\sl(2,\R))$, namely 
\begin{eqnarray}
\mathcal{Q}=\xi H,\quad \bar{\mathcal{Q}}=\eta H, \quad\mathcal{I}=i\eta\xi,
\end{eqnarray}
where $K=:q_*^H$. Then the elements 
\begin{eqnarray}
[\mathcal{Q}, E]=2E\xi, \quad [\bar{\mathcal{Q}},F]=-2F\eta, \quad K\mathcal{I},
\end{eqnarray}
up to proportionality coefficients are the generators of $\cU_q(\osp(1|2))$, namely the following Proposition holds.
\begin{Prop}\label{gen}
Given a representation for $\cU_{q_*}(\sl(2,\R))$, there exists a representation of $\cU_q(\osp(1|2))$ by
\Eq{\cE=\a E\xi, \quad \cF=F\eta, \quad \cK=Ki\eta\xi,}
where \Eq{\a=\frac{q_*-q_*\inv}{q-q\inv}=i\frac{q+q\inv}{q-q\inv}=\cot(\pi b^2)>0.}
\end{Prop}
\begin{proof}
\Eqn{
\cK\cE&=(Ki\eta\xi) (\a E\xi)\\
&=\a KEi\eta\xi\xi\\
&=\a (-q^2EK)(-\xi i\eta\xi)\\
&=q^2(\a E\xi)(Ki\eta\xi) \\
&=q^2\cE\cK,
}
\Eqn{
\cK\cF&=(Ki\eta\xi) (F\eta)\\
&=KFi\eta\xi\eta\\
&=(-q^{-2}FK)(-\eta i\eta\xi)\\
&=q^{-2}(F\eta)(Ki\eta\xi)\\
&=q^{-2}\cF\cK,
}
\Eqn{\cE\cF+\cF\cE&=(\a E\xi)(F\eta)+(F\eta)(\a E\xi)\\
&=\a( EF-FE)\xi\eta\\
&=\a\left(\frac{K-K\inv}{q_*-q_*\inv}\right)\xi\eta\\
&=\a\left(\frac{K\xi\eta-K\inv \xi\eta}{q_*-q_*\inv}\right).
}
Note that $K\xi\eta=i\cK$ and $K\inv \xi\eta=(K\eta\xi)\inv=(-i\cK)\inv=i\cK\inv$, we have
\Eqn{\cE\cF+\cF\cE&=\a\left(\frac{i\cK-i\cK\inv}{q_*-q_*\inv}\right)\\
&=i\frac{\cK-\cK\inv}{q-q\inv}.
}

Explicitly, consider the standard representation 
\Eq{U=e^{2\pi b_* x}, \tab V=e^{2\pi b_*p},} so that \Eq{UV=-q^2VU.}

Then by \eqref{efk} we have \Eq{E=i\frac{V+ZU\inv}{q_*-q_*\inv},\tab F=i\frac{U+Z\inv V\inv}{q_*-q_*\inv},\tab K=q_*\inv UV.}

Now let \Eq{U_+=U\xi,\tab V_+=V\xi,\tab U_-=U\eta,\tab V_-=V\eta.}

Then, explicitly $\cE,\cF,\cK$ are given by:
\Eq{\cE=\a E\xi = i\frac{V_++ZU_+\inv}{q-q\inv}, \tab \cF=F\eta=\frac{U_-+Z\inv V_-\inv}{q+q\inv},\tab \cK=q\inv U_+V_-.}

\end{proof}

Therefore for a modular double representation of $\cU_{q_*, \t q_*}(\sl(2,\R))$ on $L^2(\mathbb{R})$ there exists a representation of $\cU_{q,\tau(q)}(\osp(1|2))$, the modular double 
of $\cU_{q}(\osp(1|2))$ constructed by means of generators $E, F, K$ and their double $\tilde E$, $\tilde F, \tilde K$ as in Proposition \ref{gen}. We notice a new relation 
between $q$ and its dual, in this case $\tau(q)$:
\begin{eqnarray}
q=e^{\pi i(b_*^2-\frac{1}{2})}, \quad  \tau(q):=e^{\pi i(\frac{1}{{b_*}^2}-\frac{1}{2})}.
\end{eqnarray}
It is also useful to notice the following (nonstandard) commutation relations between modular double generators:
\begin{eqnarray}
&&[\mathcal{E}, \t{\mathcal{E}}]_-=0, \quad [\mathcal{F}, \t{\mathcal{F}}]_-=0,\nonumber\\
&&[\mathcal{E}, \tilde{\mathcal{F}}]_+=0, \quad [\mathcal{F}, \tilde{\mathcal{E}}]_+=0, \nonumber\\
&&[\mathcal{E}, \tilde{\mathcal{K}}]_-=0, \quad [\mathcal{F}, \tilde{\mathcal{K}}]_-=0,\nonumber\\
&&[\mathcal{K}, \tilde{\mathcal{E}}]_-=0, \quad [\mathcal{K}, \tilde{\mathcal{F}}]_-=0, 
\end{eqnarray}
where $[\cdot, \cdot]_{\mp}$ stands for commutator and anti-commutator correspondingly. Let us denote by $P^s_z$ the representation of $\cU_{q,\tau(q)}(\osp(1|2))$ on ${L_2(\mathbb{R})\otimes \mathbb{C}^{1|1}}$. The operators $\mathcal{E}, \mathcal{F}, \mathcal{K}$ as well as their modular dual counterparts are densely defined in the resulting space, namely 
we can consider as usual the action of unbounded operators $U,V$ acting on the core subspace $\cW\subset L^2(\R)$ (cf. \eqref{WW}), so that the action of operators $\mathcal{E}, \mathcal{F}, \mathcal{K}$ and $\tilde{\mathcal{E}}$, $\tilde{\mathcal{F}}$, $\tilde{\mathcal{K}}$ is defined on the dense subspace $\cW\otimes \mathbb{C}^{1|1}$.  

In the next section we give a construction of $R$-matrix acting on the tensor product $P^s_{z_1}\otimes P^s_{z_2}$. In order to do that we will introduce the following auxiliary object. 

\begin{Def}Define the transformation $\Phi, \Phi\inv: L^2(\R)\to L^2(\R)$, which preserves $\cW$, as follows:
\Eq{\Phi:& f(x)\mapsto e^{-\frac{\pi i}{4b_*^2}}e^{-\frac{\pi x}{b_*}} f(x+\frac{i}{b_*})=\tilde{q}^{-\frac{1}{4}}\tilde{U}^{-\frac{1}{2}}\tilde{V}^{-\frac{1}{2}}f\\
\Phi\inv:& f(x)\mapsto e^{-\frac{\pi i}{4b_*^2}}e^{\frac{\pi x}{b_*}} f(x-\frac{i}{b_*})=\tilde{q}^{-\frac{1}{4}}\tilde{U}^{\frac{1}{2}}\tilde{V}^{\frac{1}{2}}f}
\end{Def}
Then the following Proposition holds, which will be crucial for the definition of $R$-matrix.

\begin{Prop}\label{phi} $\Phi$ and $\Phi\inv$ gives the standard isomorphism of $\cU_{q_*}(\sl(2,\R))$:
\Eq{E\mapsto -E, \tab F\mapsto -F, \tab K\mapsto K}
\end{Prop}
\begin{proof}
It suffices to show that the action of $U\mapsto -U$ and $V\mapsto -V$.

\Eqn{\Phi\inv \circ U \circ \Phi f(x)&=\Phi\inv \circ U\left( e^{-\frac{\pi x}{b_*}}f(x+\frac{i}{2b_*})\right)\\
&=\Phi\inv e^{2\pi b_*\pi x}e^{-\frac{\pi x}{b_*}}f(x+\frac{i}{2b_*})\\
&=e^{2\pi b_*\pi (x-\frac{i}{2b_*})}f(x)\\
&=e^{-\pi i}U f(x)=-Uf(x),
}
\Eqn{\Phi\inv \circ V \circ \Phi f(x)&=\Phi\inv \circ V\left( e^{-\frac{\pi x}{b_*}}f(x+\frac{i}{2b_*})\right)\\
&=\Phi\inv e^{-\frac{\pi (x-ib_*)}{b_*}}f(x-ib_*+\frac{i}{2b_*})\\
&=e^{\frac{\pi ib_*}{b_*}}f(x-ib_*)\\
&=e^{\pi i}V f(x)=-Vf(x).
}

The case for $\Phi\inv$ is similar, with $e^{\pm\pi i}$ replaced by $e^{\mp\pi i}$ instead.

In particular, taking into account the phases, we actually have
\Eq{\Phi\inv \circ E \circ \Phi = e^{\pi i}E, \tab \Phi\inv \circ F \circ \Phi =  e^{-\pi i }F,}
\Eq{\Phi \circ E \circ \Phi\inv = e^{-\pi i }E, \tab \Phi \circ F \circ \Phi\inv =  e^{\pi i }F.}
\end{proof}

\section{$R$-matrix for $\cU_{q, \tau(q)}(\osp(1|2))$}
First we note that the expression for the $R$-matrix acting in the product of finite dimensional representations was given in \cite{kr}.  Now we will define the analogue of $R$-matrix acting in the tensor product of $P_{z_1}\otimes P_{z_2}$. 

Let $e, f$ be the elements of the modular double $\cU_{q_*,\t q_*}(\sl(2,\mathbb{R})) $ as in Section 2. 
The the following simple property holds.

\begin{Prop}\label{defgb} We have
\Eq{(\Phi^{-1}\ox 1) g_{b_*}(e\ox f)(\Phi\ox 1)=(1\ox \Phi)g_{b_*}(e\ox f)(1\ox\Phi^{-1}).}
\end{Prop}
\begin{proof}
By the action given in Prop \ref{phi}, 
$$(\Phi^{-1}\ox 1) g_{b_*}(e\ox f)(\Phi\ox 1)=g_{b_*}(e^{\pi i} e\ox f)=(1\ox {\Phi})g_{b_*}(e\ox f)(1\ox{\Phi}^{-1}),$$
where $g_{b_*}(e^{\pi i} e\ox f)$ is defined by
\Eq{g_{b_*}(e^{\pi i} X)=\int_{\R+i0} \frac{e^{-\pi it^2}e^{-\pi {b_*}\inv t}X^{ib_{*}\inv t}}{G_b(Q+it)}dt.}

In particular, since by Corollary \ref{Propgb}, $g_{b_*}(e^{\pi i}x)$ is a bounded function, the above relations which hold on a natural dense domain for the unbounded transformation $\Phi$ can be extended to the whole Hilbert space.
\end{proof}
\begin{Rem} Due to the above observation, we will \emph{define} 
\Eq{\label{bound}g_{b_*}(-e\ox f):=g_{b_*}(e^{\pi i}e\ox f)} as a bounded operator. This together with an additional factor can make the operator unitary, as being used for example in \cite{Wo} to deal with the quantum exponential function defined over general self-adjoint operators, but we will not need such generality in this paper.
\end{Rem}

Let us denote by 
\Eq{\hat{e}:=-i(q-q\inv)\cE,\tab \hat{f}:=(q+q\inv)\cF.}

Then we are ready to define the $R$-matrix operator for $\cU_{q, \tau(q)}(\osp(1|2))$. 

\begin{Def} We define the operator $R$ acting on $L^2(\R)\ox L^2(\R)\ox \C^{1 \vert 1}\otimes \C^{1\vert 1} $ by
\Eq{R=Qg_{b_*}(i\hat{e}\ox \hat{f}),}
where $Q$ is given by
\Eqn{Q&=\veca{-q_*^{H\ox H}&0&0&0\\0&q_*^{H\ox H}&0&0\\0&0&q_*^{H\ox H}&0\\0&0&0&q_*^{H\ox H}}\\
&=\frac{1}{2}(1\ox 1+i\eta\xi\ox 1+1\ox i\eta\xi+\eta\xi\ox \eta\xi)q_*^{H\ox H},
}
with $K=q_*^H$, and $g_{b_*}(i\hat{e}\ox \hat{f})$ means the following operator acting on $L^2(\R)\ox L^2(\R)\ox \C^{1 \vert 1}\otimes \C^{1\vert 1}$:
\Eqn{
g_{b_*}(i\hat{e}\ox \hat{f})&=g_{b_*}(ie\xi\ox f\eta)\\
&=Pg_{b_*}\veca{e\ox f&0&0&0\\0&e\ox f&0&0\\0&0&-e\ox f&0\\0&0&0&-e\ox f}P^*\\
&=P\veca{g_{b_*}(e\ox f)&0&0&0\\0&g_{b_*}(e\ox f)&0&0\\0&0&g_{b_*}(-e\ox f)&0\\0&0&0&g_{b_*}(-e\ox f)}P^*\\
&=P\Phi_1 g_{b_*}(e\ox f) \Phi_1\inv P^*,
}
where $\Phi_1=diag(1\ox 1,1\ox 1, \Phi\inv\ox 1,\Phi\inv\ox 1)$.

By Prop \ref{defgb}, it also equals to
\Eq{R=P\Phi_2' g_{b_*}(e\ox f) {\Phi_2'}\inv P^*,}
where
$\Phi_2'=diag(1\ox 1,1\ox 1, 1\ox \Phi,1\ox \Phi)$. Moreover, by \eqref{bound}, $R$ is a bounded operator.

\end{Def}

Let us prove all the necessary properties it has to satisfy.

\begin{Thm}\label{Rbraid} $R$ satisfies the braiding relation $\Delta'(\cdot){R}=R\Delta(\cdot)$.
\end{Thm}
First recall from Lemma \ref{rank1prop} that in $\cU_q(\sl(2,\R))$, using the relations such as
\Eqn{(K\ox E)q^{\frac{H\ox H}{2}}&=q^{\frac{H\ox H}{2}}(1\ox E),\\
(E\ox 1)q^{\frac{H\ox H}{2}}&=q^{\frac{H\ox H}{2}}(E\ox K\inv),}
the braiding relation is equivalent to the following relations:
\Eqn{(1\ox E+E\ox K\inv)g_b(e\ox f)&=g_b(e\ox f)(1\ox E+E\ox K),\\
(F\ox 1 +K\ox F)g_b(e\ox f)&=g_b(e\ox f)(F\ox 1+K\inv\ox F).}
\begin{Lem}We have 
\Eqn{(\cK\ox \cE)Q=Q(1 \ox \cE),&\tab (\cE\ox 1)Q=Q(\cE\ox \cK\inv),\\
(\cF\ox \cK\inv)Q=Q(\cF \ox 1),&\tab (1\ox \cF)Q=Q(\cK\ox \cF).
}
\end{Lem}
\begin{proof}For the first line, it suffices to show 
$$(i\eta\xi\ox \xi)(1\ox 1+i\eta\xi\ox 1+1\ox i\eta\xi+\eta\xi\ox \eta\xi)=(1\ox 1+i\eta\xi\ox 1+1\ox i\eta\xi+\eta\xi\ox \eta\xi)(1\ox \xi)$$
and
$$(\xi\ox 1)(1\ox 1+i\eta\xi\ox 1+1\ox i\eta\xi+\eta\xi\ox \eta\xi)=(1\ox 1+i\eta\xi\ox 1+1\ox i\eta\xi+\eta\xi\ox \eta\xi)(\xi\ox i\eta\xi).$$
In other words, 
\Eqn{\veca{0&-1&0&0\\-1&0&0&0\\0&0&0&-1\\0&0&-1&0}\veca{-1&0&0&0\\0&1&0&0\\0&0&1&0\\0&0&0&1}&=\veca{-1&0&0&0\\0&1&0&0\\0&0&1&0\\0&0&0&1}\veca{0&1&0&0\\1&0&0&0\\0&0&0&-1\\0&0&-1&0},\\
\veca{0&0&1&0\\0&0&0&1\\1&0&0&0\\0&1&0&0}\veca{-1&0&0&0\\0&1&0&0\\0&0&1&0\\0&0&0&1}&=\veca{-1&0&0&0\\0&1&0&0\\0&0&1&0\\0&0&0&1}\veca{0&0&-1&0\\0&0&0&1\\-1&0&0&0\\0&1&0&0}.
}
The second line for $\cF$ is completely analogous with $\xi$ replaced by $\eta$.

\end{proof}
\begin{proof}[Proof of Theorem \ref{Rbraid}] Using the above Lemma, we have 
\Eqn{\D'(\cE)R&=(\cK\ox \cE+\cE\ox 1)Qg_{b_*}(i\hat{e}\ox \hat{f})\\
&=Q(1\ox \cE+\cE\ox \cK\inv)g_{b_*}(i\hat{e}\ox \hat{f})\\
&=Q(1\ox E\xi+E\xi\ox K\inv i\eta\xi)g_{b_*}(ie\xi\ox f\eta)\\
&=Q(1\ox E+iE\xi\ox K\inv \eta)(1\ox \xi)g_{b_*}(ie\xi\ox f\eta)\\
&=Q(1\ox E+iE\xi\ox K\inv \eta)g_{b_*}(ie\xi\ox f\eta)(1\ox \xi)\\
(*)&=Qg_{b_*}(ie\xi\ox f\eta)(1\ox E\xi+iE\xi\ox K \eta)(1\ox \xi)\\
&=Qg_{b_*}(ie\xi\ox f\eta)(1\ox E\xi+E\xi\ox K i\eta\xi)\\
&=Qg_{b_*}(ie\xi\ox f\eta)(1\ox \cE+\cE\ox \cK)\\
&=R\D(\cE).
}
Therefore we have to prove $(*)$:
$$(1\ox E+iE\xi\ox K\inv \eta)g_{b_*}(ie\xi\ox f\eta)=g_{b_*}(ie\xi\ox f\eta)(1\ox E\xi+iE\xi\ox K \eta),$$
in other words, by diagonalizing $\xi\ox \eta$,
$$(1\ox E\pm E\ox K\inv)g_{b_*}(\pm e\ox f)=g_{b_*}(\pm e\ox f)(1\ox E\pm E\ox K).$$
The $+$ equation is the standard one by Lemma \ref{rank1prop}. For the $-$ equation, taking into account the definition of $g_{b_*}(- e\ox f)=(\Phi\ox1) (g_{b_*}(e\ox f))$, by applying $(\Phi\ox 1)$ to the equations, this is again the standard relation.

Finally the proof for $\D(\cF)$ is completely analogous.
\end{proof}

In order to prove the quasi-triangularity relations, one has to define consistently what $(\D\ox id)(R)$ and $(id\ox \D)(R)$ mean. The necessary ingredient for those is the appropriate definition of $g_{b_*}(\pm U\pm V)$, where $UV=q_*^2VU$. Consider the following relation between bounded operators
$$g_{b_*}(U+zV)=g_{b_*}(U)g_{b_*}(zV),$$
where $z\in\R_{>0}$, by the quantum exponential relation \eqref{qsum0}. By Corollary \ref{Propgb}, $g_{b_*}(zV)$ is bounded for $\arg z>0$ and the action on $f(x)$ depends analytically on z. Hence the right hand side gives the expression for the analytic continuation of the operator with respect to $z$. In particular, by analytic continuation to $z=e^{\pi i}$, we can define
\Eq{\label{gbu-v}g_{b_*}(U-V):=g_{b_*}(U)g_{b_*}(e^{\pi i}V)=g_{b_*}(U)g_{b_*}(-V).}
Similarly we define
\Eq{\label{gb-u+v}g_{b_*}(-U+V):=g_{b_*}(e^{\pi i}U)g_{b_*}(V)=g_{b_*}(-U)g_{b_*}(V),}
while $g_{b_*}(-U-V)$ is simply $g_{b_*}(e^{\pi i}(U+V))$ as before.

Using this notion we are ready to formulate the following Theorem.
\begin{Thm}\label{RqT} $R$ satisfies the quasi-triangular relations $(\D \ox id)(R)=R_{13}R_{23}$ and 
$(id\ox\D)(R)=R_{13}R_{12}$.
\end{Thm}
\begin{Lem} We have 
\Eq{Q_{13}Q_{23}=\D(Q).}
\end{Lem}
\begin{proof}
Since $$\D(\cK)=\cK\ox \cK= Ki\eta\xi \ox Ki\eta\xi =(K\ox K) (i\eta\xi\ox i\eta\xi),$$
we see that both $K$ and $i\eta\xi$ is group-like. Hence
$$\D(Q)=\frac{1}{2}(1\ox1 \ox 1+i\eta\xi\ox i\eta\xi\ox 1+1\ox1\ox i\eta\xi-i\eta\xi\ox i\eta\xi\ox i\eta\xi)q_*^{H\ox 1\ox H+1\ox H\ox H}.$$
Then $\D(Q)=Q_{13}Q_{23}$ amounts to the matrix equation
$$diag(1,1,-1,1,-1,1,1,1)=diag(-1,1,-1,1,1,1,1,1)diag(-1,1,1,1,-1,1,1,1).$$
\end{proof}

\begin{proof} [Proof of Theorem \ref{RqT}] 
We have
\Eqn{&R_{13}R_{23}\\
&=Q_{13}g_{b_*}(i\hat{e}\ox1\ox \hat{f})Q_{23}g_{b_*}(1\ox i\hat{e}\ox \hat{f})\\
&=Q_{13}Q_{23}g_{b_*}(i\hat{e}\ox\cK\ox \hat{f})g_{b_*}(1\ox i\hat{e}\ox \hat{f})\\
&=\D(Q)g_{b_*}(i\hat{e}\ox\cK\ox \hat{f})g_{b_*}(1\ox i\hat{e}\ox \hat{f}).
}
Therefore we have to show that 
$$g_{b_*}(i\hat{e}\ox\cK\ox \hat{f})g_{b_*}(1\ox i\hat{e}\ox \hat{f})\\=g_{b_*}(i\D(\hat{e})\ox \hat{f}).$$
Using our transformations, we have
\Eqn{&g_{b_*}(i\hat{e}\ox\cK\ox \hat{f})\\
&=P_{13}g_{b_*}\left((ie\ox K\ox f)\cdot diag(-1,-1,-1,-1,1,1,1,1)\right) P_{13}^*
}
and
\Eqn{&g_{b_*}(1\ox i\hat{e}\ox \hat{f})\\
&=P_{23}g_{b_*}\left((i\ox e\ox f)\cdot diag(1,1,-1,-1,1,1,-1,-1)\right)P_{23}^*,
}

where
\Eqn{ \tiny{
P_{13}=\frac{1}{\sqrt{2}}\veca{
-1&0&0&0&1&0&0&0\\
0&1&0&0&0&-1&0&0\\
0&0&-1&0&0&0&1&0\\
0&0&0&1&0&0&0&-1\\
0&1&0&0&0&1&0&0\\
1&0&0&0&1&0&0&0\\
0&0&0&1&0&0&0&1\\
0&0&1&0&0&0&1&0
},
P_{23}=\frac{1}{\sqrt{2}}\veca{
-1&0&1&0&0&0&0&0\\
0&1&0&-1&0&0&0&0\\
0&1&0&1&0&0&0&0\\
1&0&1&0&0&0&0&0\\
0&0&0&0&-1&0&1&0\\
0&0&0&0&0&1&0&-1\\
0&0&0&0&0&1&0&1\\
0&0&0&0&1&0&1&0
}
}.}
Now note that $$P_{13}^*P_{23}=cP_{23}P_{13}^*,$$
where $$c=\tiny{\veca{0&0&0&0&-1&0&0&0\\0&0&0&0&0&-1&0&0\\0&0&0&0&0&0&1&0\\0&0&0&0&0&0&0&1\\
-1&0&0&0&0&0&0&0\\0&-1&0&0&0&0&0&0\\0&0&1&0&0&0&0&0\\0&0&0&1&0&0&0&0}},$$
$$diag(-1,-1,-1,-1,1,1,1,1)\cdot c=c \cdot diag(1,1,1,1,-1,-1,-1,-1),$$
$$diag(1,1,1,1,-1,-1,-1,-1)\cdot P_{23}=P_{23}\cdot diag(1,1,1,1,-1,-1,-1,-1),$$
$$P_{13} cP_{23}=P_{23}P_{13},$$
and finally
$$P_{13}^*\cdot diag(1,1,-1,-1,1,1,-1,-1)=diag(1,1,-1,-1,1,1,-1,-1)\cdot P_{13}^*.$$

Combining, we get
\Eqn{R_{13}R_{23}=&\D(Q)P_{23}P_{13} XP_{13}^*P_{23}^*,
}
where
\Eqn{X=&g_{b_*}((e\ox K\ox f)\cdot diag(1,1,1,1,-1,-1,-1,-1))\cdot\\
&g_{b_*}((1\ox e\ox f)\cdot diag(1,1,-1,-1,1,1,-1,-1))\\
&=diag(g_{++},g_{++},g_{+-},g_{+-},g_{-+},g_{-+},g_{--},g_{--})
}
and $$g_{\pm,\pm}=g_{b_*}(\pm e\ox K\ox f)g_{b_*}(\pm 1\ox e\ox f),$$ which are bounded operators.

Recall that due to conjugation by $\Phi$, the above phase is chosen to be $-1=e^{\pi i}$. Hence according to our definition using \eqref{gbu-v}-\eqref{gb-u+v}, 
$$g_{\pm,\pm}=g_{b_*}(\pm e\ox K\ox f\pm 1\ox e\ox f).$$
One can check that this is indeed compatible with our previous definition using $\Phi$. For example
\Eqn{g_{+,-}&=g_{b_*}(e\ox K\ox f)g_{b_*}(-1\ox e\ox f)\\
&:=g_{b_*}(e\ox K\ox f)(1\ox\Phi\inv\ox1)g_{b_*}(1\ox e\ox f)(1\ox\Phi\ox1)\\
&=(1\ox \Phi\inv\ox 1)g_{b_*}(e\ox K\ox f)g_{b_*}(1\ox e\ox f)(1\ox\Phi\ox1)\\
&=(1\ox \Phi\inv\ox 1)g_{b_*}(e\ox K\ox f+1\ox e\ox f)(1\ox\Phi\ox1)\\
&=g_{b_*}(e\ox K\ox f-1\ox e\ox f).
}

Finally, a simple computation shows that
\Eqn{&P_{23}P_{13} XP_{13}^*P_{23}^*\\
&=:g_{b_*}\tiny{\veca{
0&0&0&-1\ox e\ox f&0&e\ox K\ox f&0&0\\
0&0&1\ox e\ox f&0&-e\ox K\ox f&0&0&0\\
0&1\ox e\ox f&0&0&0&0&0&e\ox K\ox f\\
-1\ox e\ox f&0&0&0&0&0&-e\ox K\ox f&0\\
0&-e\ox K\ox f&0&0&0&0&0&-1\ox e\ox f\\
e\ox K\ox f&0&0&0&0&0&1\ox e\ox f&0\\
0&0&0&-e\ox K\ox f&0&1\ox e\ox f&0&0\\
0&0&e\ox K\ox f&0&-1\ox e\ox f&0&0&0\\}
}\\
&=g_{b_*}((e\ox K\ox f)(\xi\ox i\eta\xi\ox i \eta)+(1\ox e\ox f)\cdot(1\ox \xi\ox i\eta)\\
&=g_{b_*}(i\hat{e}\ox \cK\ox \hat{f}+i\ox \hat{e}\ox \hat{f})\\
&=g_{b_*}(i\D(\hat{e})\ox \hat{f})
}
as desired.

Similarly, for $(1\ox \D)(R)$, we just need to use the definition of $g_{b_*}$ with $1\ox\Phi\inv$ in place of $\Phi\ox1$.
\end{proof}

\begin{Rem} We note here that operator $R$ is bounded but not unitary.  Its inverse is given by the (unbounded) operator $\tilde{R}$ which can be constructed in a similar 
fashion from $g_{b^*}$ using $(\Phi\inv\ox1)$ instead of $(\Phi\ox1)$. A calculation shows that 
\Eq{R^*\tilde{R}=1.}
\end{Rem}

\section{Final remarks}

First we want to underline the following fact about the representations $P^s_z$ we introduce. In the standard case of finite-dimensional representations of 
$U_q(\osp(1|2))$ in the classical limit $q\to 1$ the representation space of $V^s_{\lambda}$, labeled by heights weight $\lambda\in\mathbb{Z}$ decomposes in the following 
way $V^s_{\lambda}=V^e_{\lambda}\oplus V^o_{\lambda-1}$ with respect to $\cU(\sl(2))$ subalgebra, where $V_{\lambda}$ are $\cU(\sl(2))$ heights weight modules, so that $e$ and $o$ letters mean that decomposition preserves the grading, i.e. $V^e_{\lambda}$ is spanned by even vectors and $V^o_{\lambda-1}$ by odd ones. In the case of 
$P^s_z$, it naturally decomposes in the following way: $P^s_z=P^{(e)}_z\oplus P^{(o)}_z$, where $P^{(e)}_z$,  $P^{(o)}_z$ are representations of 
$\cU_{q_*\tilde{q_*}}(\sl(2,\R))$ generated by even and odd vectors correspondingly. So $\sl(2,\mathbb{R})$ makes its appearance similar to the classical case (although it is not a subalgebra anymore). 

An important problem is the construction of the tensor category of representations $P^s_z$. 
It is known that the representations of the modular double of $\cU_q(\sl(2),\R)$ form a "continuous" tensor category. 
We hope that the same holds for the representations of $\cU_q(\osp(1|2, \mathbb{R}))$, which we introduced in this article form a similar tensor category as well as in the case of finite-dimensional representations.

After construction of this category it will be interesting to compare the corresponding $3j$ symbols to the 
coefficients in the fusion product of $\cN=1$ SUSY Liouville vertex operators, following the ideas from \cite{teschner}.

Another interesting problem is to generalize the construction of modular double and its universal $R$-matrix to higher rank simple superalgebras. In the compact case, explicit formula for the universal $R$-matrix of all simple superalgebra using the $q$-exponential function has been obtained in \cite{KT}, while construction of the modular double $\cU_{q\tilde{q}}(\g_\R)$ using positive representations was done for all simple algebras \cite{fip}, \cite{Ip2}, \cite{Ip3}, and their corresponding $R$-operator in \cite{Ip4}.

\section*{Acknowledgments}
We are grateful to Hyun-Kyu Kim and Rishi Raj for useful discussions. The first author is supported by World Premier International Research Center Initiative (WPI Initiative), MEXT, Japan.

\end{document}